\documentclass{amsart}

\usepackage{graphicx}

\usepackage{amsmath,amssymb}
\usepackage{amsthm}
\usepackage{amsfonts}
\newtheorem{df}{Definition}[section]
\newtheorem{thm}[df]{Theorem}
\newtheorem{lemm}[df]{Lemma}

\newtheorem{rem}[df]{Remark}

\newcommand{\aug}{\textit{aug}}
\newcommand{\id}{\textit{id}}

\begin{document}

\title{The logarithms of Dehn twists on non-orientable surfaces}
\author{Tsuji Shunsuke}
\maketitle

\begin{abstract}
We introduce a Lie algebra associated with a non-orientable surface, which is an analogue for the Goldman Lie algebra of an oriented surface. As an application, we deduce an explicit formula of the Dehn twist along an annulus simple closed curve on the surface as in Kawazumi-Kuno  \cite{Kawazumi} \cite{KK} and Masseyeau-Turaev \cite{MT}.
\end{abstract}
\section{Introduction}

In study of an oriented surfaces, the Goldman Lie algebra
plays an important role. Goldman \cite{Goldman} defined a Lie bracket 
on the free $\mathbb{Z}$-module with basis the set
of conjugacy classes in the fundamental group of the oriented
surface. The bracket corresponds to the Poisson bracket of
smooth functions on the representation space of the fundamental group.
This is called the Goldman Lie bialgebra. 
Turaev \cite{Turaev} found that the Goldman Lie bialgebra has the structure
of a Lie algebra, which is called the Goldman-Turaev Lie algebra.
Furthermore Turaev \cite{Turaev} showed that the skein algebra of links
in the cylinder over an oriented surface quantizes
the Goldman-Turaev Lie bialgebra on the surface.
Gadgil \cite{Gadgil} showed that a homotopy equivalence between compact
oriented surfaces with non-empty boundary is homotopic to a
homeomorphism if and only if it commutes with the
Goldman Lie bracket. Kawazumi and Kuno \cite{Kawazumi} found that Goldman Lie algebra acts on
the group ring of fundamental group and that the action
induces more detailed structures on the Goldman Lie algebra.

On a non-orientable surface the local intersection number can
be defined only over fields whose characteristic is 2.
For example, in Kawazumi-Kuno \cite{Kawazumi} \cite{KK} and Massuyeau-Turaev \cite{MT}
the quantity $\frac{1}{2}(\log (c))^2$ gives the logarithm of the Dehn twist along the simple closed curve $c$ on an oriented surface,
but we can't define it on a non-orientable surface.
In this paper, we introduce a Lie subalgebra of the Goldman Lie algebra on the orientation cover with coefficients in a commutative ring containing the rationals $\mathbb{Q}$ and
the action of the Lie subalgebra on the group ring of fundamental
group of the non-orientable surface. The action can be quantized in the sense of
Turaev \cite{Turaev}. As an application, we deduce an explicit
formula of the Dehn twists along an annulus simple closed curve on the surface as in 
Kawazumi and Kuno \cite{Kawazumi}.
On a non-orientable surface, the annulus simple closed curve Dehn twists generate
the subgroup of the mapping class group
consisting of elements whose determinant of the action on the first homology group is 1, as was proved by Lickorish \cite{Lickorish}.

We conclude the introduction by fixing our notation.
Let $I$ denote denote the unit interval $[0,1]$ as usual.
Let $F$ denote a non-orientable compact connected surface with non-empty boundary, and $K$ a commutative associative ring containing the field of rationals $\mathbb{Q}$.
We define $p:\tilde{F} \to F$ to be the orientation cover of $F$ as shown in Figure \ref{orientation:cover}.
We fix the curve segments $\delta_1,\delta_2 \dots \delta_n ,\delta'_1 ,\delta'_2 \dots \delta'_n $ in $\tilde{F}$ as in Figure \ref{orientation:cover} such that $p(\delta_i)=p(\delta' _i)$ for all $i=1,2 \dots n$. We fix an orientation of the surface $\tilde{F}$. The surface $F^u$ denotes the left connected component of $F \backslash \cup^n_{i=1}(\delta_i \cup \delta'_i)$ and the surface $F^d$ the right in Figure \ref{orientation:cover}. For $x_0 \in F \backslash \cup^n_{i=1} p(\delta_i) $, the fiber $p^{-1}(x_0)$ consists of $x^u_0 \in F^u$ and $x_0^d \in F^d$. We identify $F \backslash \cup^n_{i=1} p(\delta_i) $ and $F^u$ as oriented surfaces.

\begin{figure}
\includegraphics[scale=0.7]{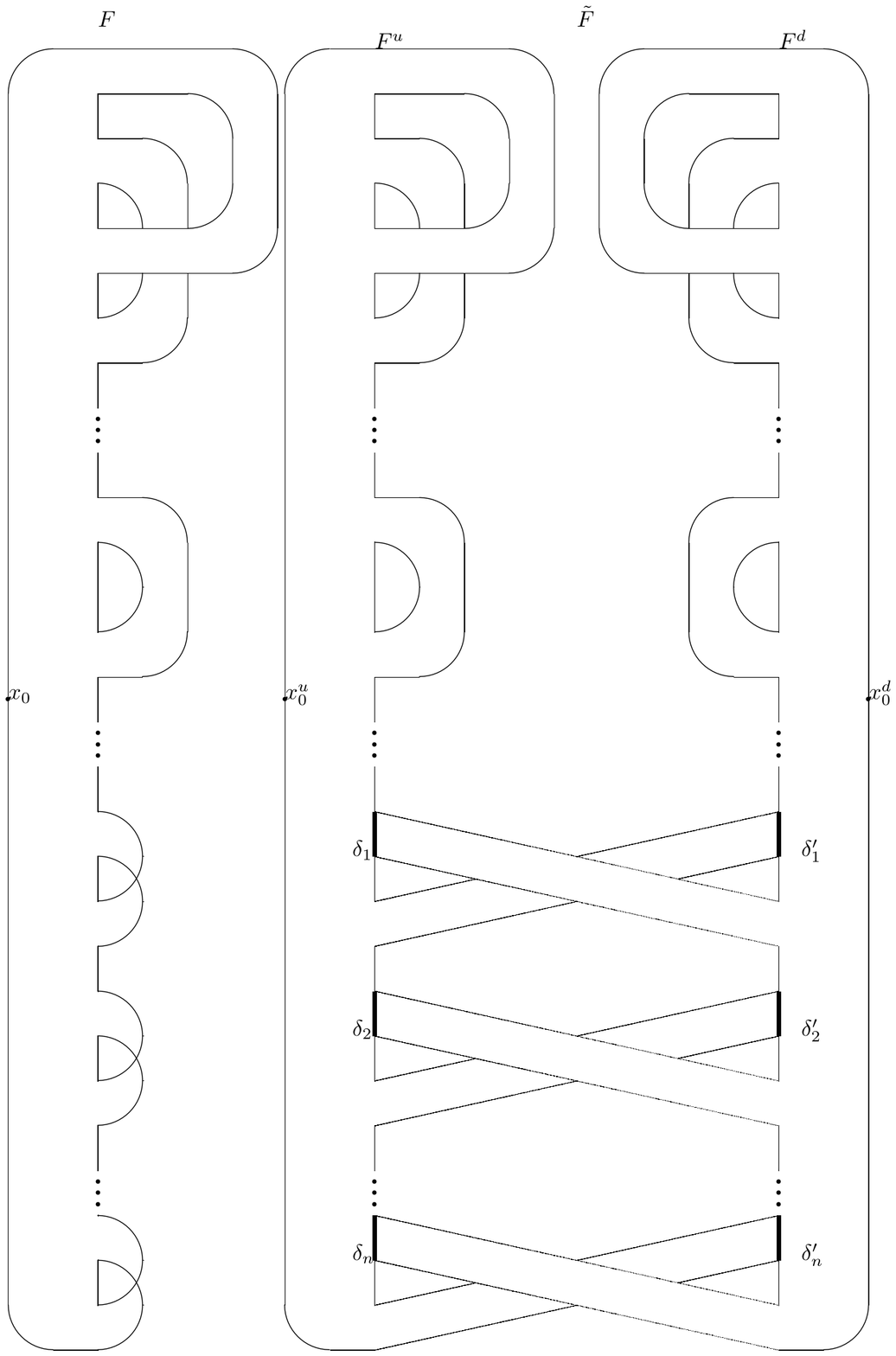}
\caption{The orientation cover} \label{orientation:cover}
\end{figure}

\subsection*{Acknowledgment}
The author thanks his adviser, Nariya Kawazumi, for helphul discussion.

\section{The action on the fundamental group}

Let $\pi = \pi_1(F,x_0)$ be the fundamental group of $F$ with basepoint $x_0$, 
and $\hat{\pi} = \hat{\pi}_1(\tilde{F})$ the set of free homotopy classes of oriented loops in $\tilde{F}$. $K\pi $ denotes the group ring of $\pi$ over $K$, $K \hat{\pi}$ the free $K$-module with basis $\hat{\pi}$. 
Let $\tau:\tilde{F} \to \tilde{F}$ be the unique nontrivial covering transformation of $p$. 
Representatives of $x \in \pi$ and $y \in \hat{\pi}$ are called in general position if $x \cup p(y) \cup (\cup^n_{i=1} p(\delta_i)):I \cup S^1 \cup (\cup_{i=1}^n I_i) \to F$ is an immersion with at worst transverse double points.

\begin{df}[see Kawazumi and Kuno \cite{Kawazumi} Definition 3.2.1.]

For $x \in \pi $ and $y \in \hat{\pi}$ ,we define the action $\tilde{\sigma}(y)(x) =y(x)\in K \pi$ by the following formula, where we choose representatives of $x \in \pi$ and $y \in \hat{\pi}$ in general position,

\begin{multline}
\tilde{\sigma}(y)(x)=y(x) =\frac{1}{2}(\sum_{q \in p(y \cap F^u) \cap x} \varepsilon (q,p(y),x)x_{x_0q}(p(y))_q x_{qx_0} \\
-\sum_{q \in p(y \cap F^d) \cap x} \varepsilon (q,p(y),x)x_{x_0q}(p(y))_q x_{qx_0}).
\label{funca}
\end{multline}

Here $\varepsilon(q,p(y),x)$ is the local intersection number of $p(y)$ and $x$ at $q$ in $F \backslash \cup^n_{i=1} p(\delta_i) $, $(p(y))_q \in \pi_1(F,q)$ is the oriented based loop $p(y)$ based at $q$, $x_{x_0q}$ is the path along $x$ from $q$ to $x_0$, and  $x_{qx_0}$ is the path along $x$ from $x_0$ to $q$.

\end{df}

\begin{lemm}
For any $x \in \pi$ and $y \in \hat{\pi}$,the action $y(x) \in K \pi$ is well-defined.

\end{lemm} 

This lemma is proved in Lemma \ref{lift}. We can define the action $\tilde{\sigma}$ by Lemma \ref{lift} but here we take this geometrical definition.

\begin{rem}
The action $\tilde{\sigma}$ can be quantized in the  sense of Turaev \cite{Turaev}. We define the homeomorphism $\rho : I \to I$ by $ \rho(t)=1-t$. We define the oriented 3-manifold $E$ by the quotient space of $\tilde{F} \times I$ under the equivalent relation, $(\tilde{x},t) \sim (\tilde{x}',t')$ if and only if $ \tilde{x}=\tilde{x}', t=t'$ or $\tilde{x}=\tau(\tilde{x}') , t=\rho(t')$. Let $\xi$ be the quotient map $\tilde{F} \times I \to E$. The continuous map $q :E \to F, (\tilde{x},t) \mapsto p(\tilde{x})$ is the $I$-bundle map. We define the submanifold $\tilde{E} \subset E$ by $\xi(\tilde{F} \times ([0,1/3] \cup[2/3,1])$. The continuous map $\tilde{q}:\tilde{E} \to \tilde{F},(\tilde{x},t) \mapsto \tilde{x}$ for $t \in [2/3,1],  (\tilde{x},t) \mapsto \tau(\tilde{x})$ is the trivial $[0,1/3]$-bundle. We denote the Turaev skein algebra of oriented links in $\tilde{E}$ by $\mathcal{A}(\tilde{E})$ and the Turaev skein module of oriented tangles from $\xi(x^u_0,0)$ to $\xi(x^u_0,1)$ in $E$ by $\mathcal{B}(E,x_0)$. For details, see Turaev \cite{Turaev}. The action $\tilde{\sigma}$ on $K\pi$ of $K\hat{\pi}$ can be quantized by the action on  $\mathcal{B}(E,x_0)$ of  $\mathcal{A}(\tilde{E})$ in the sense of Turaev \cite{Turaev}.                                       
\end{rem}

We extend the action $\tilde{\sigma}$ by linearity to a bilinear map $K \hat{\pi} \times K \pi \to K \pi $. Let $[\ \ ,\ \ ]$ be the Lie bracket in the Goldman Lie algebra $K\hat{\pi}$. We remark that the group ring $K \pi$ is not a $K \hat{\pi}$-module with the action $\tilde{\sigma}$, see Lemma \ref{lift} and Theorem \ref{sayou}. Here $\theta$ denotes the map $\frac{1}{2}(\id -\tau ):K \hat{\pi} \to K \hat{\pi}$.

For $y,y_1,y_2 \in K \hat{\pi}$ and $x,x_1,x_2 \in K \pi $, it is easy to show the following

\begin{align*}
&y(x_1 \cdot x_2)=y(x_1)x_2+x_1y(x_2), \\
&y(x)= - \tau (y)(x), \\
&\tau([y_1,y_2] )=-[\tau(y_1),\tau(y_2)],\\
&\theta^2(y) =\theta(y) , \\
&[\theta(y_1),\theta(y_2)]=\theta([y_1,\theta(y_2)])=\theta([\theta(y_1),y_2])=\theta([\theta(y_1),\theta(y_2)]).
\end{align*}

The following lemma is easy to prove but is essential.
\begin{lemm}
\label{lift}
For $r \in \pi $, let $\tilde{r}$ be the lift of $r$ to $\tilde{F}$ starting at $x_0^u$.
It is satisfies that $\tilde{r} \in \pi_1(\tilde{F},x_0^u)$ or $\tilde{r} \in \pi_1(\tilde{F},x_0^u,x_0^d)$. For $y \in \hat{\pi}$, we have

\begin{equation}\notag
y(r)= \tilde{\sigma}(y)(r) =p(\sigma(\theta(y))(\tilde{r})),
\end{equation}

where the action $\sigma$ is defined in Kawazumi and Kuno \cite{Kawazumi} Definition 3.2.1.

\end{lemm}

The following lemma is proved by Lemma \ref{lift} and \cite{Kawazumi} Proposition 3.2.2.

\begin{lemm}

For $a,b \in \hat{\pi}$ and $r \in \pi$, we have

\begin{equation}\notag
\begin{split}
a(b(r))-b(a(r)) = ([\theta(a).b])(r) 
=([a,\theta(b)])(r)=([\theta(a),\theta(b)])(r).
\end{split}
\end{equation}

\end{lemm}

\begin{proof}Let $\tilde{r}$ be the lift of $r$ to $\tilde{F}$ starting at $x_0^u$.

By Lemma \ref{lift} and \cite{Kawazumi} Proposition 3.2.2, we have
\begin{equation*}
\begin{split}
a(b(r))-b(a(r)) &=p(\sigma(\theta(a))(\sigma(\theta(b))(\tilde{r}))-\sigma(\theta(b))(\sigma(\theta(a))(\tilde{r}))=p(\sigma([\theta(a),\theta(b)])(\tilde{r})) \\
&=p(\sigma(\theta([\theta(a),b]))(\tilde(r)))=([\theta(a).b])(r).
\end{split}
\end{equation*}

\end{proof}

\begin{lemm}

$\theta K \hat{\pi} \subset K \hat{\pi}$ is a Lie subalgebra of $K \hat{\pi}$.

\end{lemm}

\begin{proof}It suffices to check the equation $[\theta(a),\theta(b)]=\theta([\theta(a),\theta(b)]$.
\end{proof}

\begin{thm}
\label{sayou}
$K \pi$ is $\theta K\hat{\pi}$-module with $\tilde{\sigma}$.
\end{thm}

\begin{proof}For $a,b \in K\hat{\pi},r \in K\pi$, we have

\begin{equation}\notag
\theta(a)(\theta(b)(r))-\theta(b)(\theta(a)(r))=[\theta^2(a),\theta^2(b)](r)=[\theta(a),\theta(b)](r).
\end{equation}

\end{proof}

\section{Completion}

\subsection{Completion of the Goldman Lie algebra}

The groups $\pi = \pi_1(F,x_0)$ and $\tilde{\pi} =\pi_1(\tilde{F},x_0^u)$ are free groups.
Let $K\tilde{\pi}$ be the group ring of $K\tilde{\pi}$ over $K$. Let $c:K\tilde{\pi} \to K\hat{\pi}$ be the forgetful map of the basepoint $x^u_0$. Then $c$ is surjective.

Let $G$ be a free group of finite rank and $KG$ the group ring of $G$ over $K$. Define a $K$-algebra homomorphism $\aug:KG \to K$ by $g \in G \mapsto 1$.
We define $I^0G=KG$ and $I^jG=(\ker \aug)^j$.

It is well-known that $\cap^{\infty}_{j=0}I^jG=0$. See, for example, Bourbaki \cite{Bourbaki} Exercise 4.6. Furthermore we have $\cap^{\infty}_{j=0}c(I^j \tilde{\pi})=0$ by Kawazumi and Kuno \cite{KK} Corollary 4.3.2. We define the completed group ring $\widehat{KG}=\underleftarrow{\lim}_{i \rightarrow \infty}KG/(I^iG)$ and the completed Goldman Lie algebra $\widehat{K \hat{\pi}}=\underleftarrow{\lim}_{i \rightarrow \infty}K\hat{\pi}/(c(I^i\tilde{\pi})+K 1)$. 

As is proved in \cite{KK} Theorem 4.1.1 and Lemma 2.5, we have $y_i(x_j) \in I^{i+j-2} \pi$  for $y_i \in c(I^i \tilde{\pi})$ and $x_j \in I^j \pi$. Moreover we have $ \tilde{\sigma}(1)=0$. Hence $\widehat{K \hat{\pi}}$ acts $\widehat{K \pi}$ continuously as derivations.

\subsection{Dehn twist on unoriented surfaces}

We orient $S^1 =\mathbb{R}/\mathbb{Z}$ as the quotient of the line $\mathbb{R}$. We orient the annulus $S^1 \times I$ as a product manifold.

Let $t:S^1 \times I \to S^1 \times I$ be the (right handed) Dehn twist of an annulus given by the formula $S^1 \times I \to S^1 \times I,(s,t) \mapsto (s+t,t)$.
An simple closed curve is called an annulus simple closed curve if its tubular neighborhood is homeomorphic to an annulus. Let $A$ be  an oriented tubular neighborhood of an annulus simple closed curve. We define $t_A:F \to F$ by
\begin{equation*}
t_A(p)=\begin{cases}
t(p) &\text{for $p \in A$} \\
p & \text{for $p \in F \backslash A$}.
\end{cases}
\end{equation*}

We define $\log(t_A):\widehat{K\pi} \to \widehat{K \pi}$ by
\begin{equation*}
\log(t_A)(r)=-\sum_{i=1}^{\infty}\frac{1}{i}(1-t_A)^i(r).
\end{equation*}

Annulus simple closed curves can be lifted to the orientation cover. Let $l$ be a simple closed curve in the surface. If $p(l)$ is simple, $p(l)$ is an annulus circle.
\begin{thm}
\label{mainthm} Let $r$ be an element of $\tilde{\pi}$ such that an embeding of $S^1$ on $F$ represents $p(c(r))$. Then $p(c(r))$ is an annulus simple closed curve. We orient the tubular neighborhood $U$ of $c(r)$ as a submanifold of $\tilde{F}$ and the orientation of $U$ induces that on $p(U)$. We call the oriented annulus $p(U)$ of $p(c(r))$ $A$. We define the element $L \in \theta \widehat{K \hat{\pi}}$ such that $L=\theta (c((\log(r))^2))$. Then we have

\begin{equation}
\log(t_A)(\cdot)=\tilde{\sigma}(L)(\cdot):\widehat{K\pi} \to \widehat{K \pi}.
\end{equation}

In other words, if we define $e^{\tilde{\sigma}(L)} =\sum_{k=0}^\infty \frac{1}{k!} (\tilde{\sigma}(L))^k$, then 

\begin{equation*}
t_A=e^{\tilde{\sigma}(L)}:\widehat{K\pi} \to \widehat{K \pi}.
\end{equation*}

\end{thm}

\begin{proof}We use Theorem5.2.1 \cite{KK}.
Notice that the following diagram is commutative.

\begin{center}
\begin{tabular}{ccc}
$\tilde{F}$ & $\stackrel{t_{c(r)}t^{-1}_{\tau (c(r))}}{\to}$ & $\tilde{F}$ \\
$\downarrow$ & $\circlearrowright$ & $\downarrow$ \\
$F$ & $\stackrel{t_A}{\to}$ & $F$
\end{tabular}
\end{center}

Notice that $t_{c(r)}$ and $t_{\tau(c(r))}$ are commutative.
Let $\tilde{x}$ be the lift of $x \in \pi $ to $\tilde{F}$ and we have

\begin{equation}\notag
\begin{split}
\log (t_A) (x) &=p(\log(t_{c(r)}t^{-1}_{\tau(c(r))})(\tilde{x})) \\
&=p((\log (t_{c(r)})-\log (t_{\tau(c(r))}))(\tilde{x})) \\
&= p((\sigma(c(\frac{1}{2}(\log(r))^2))-\sigma(c(\frac{1}{2}(\log(\tau(r)))^2)))(\tilde{x})) \\
&=\tilde{\sigma}(\theta(c((\log(r))^2)))(x).
\end{split}
\end{equation}

Furthermore we have

\begin{equation*}
e^L(x) =e^{\log t_A}(x) =p(e^{\log(t_{c(r)}t^{-1}_{\tau(c(r))})}(\tilde{x})) 
=p(t_{c(r)}t^{-1}_{\tau(c(r)}(\tilde{x})) =t_A(x).
\end{equation*}

\end{proof}

\begin{rem}
For any embeding oriented annulus $A$ in the surface $F$, there exists $r \in \tilde{\pi}$ such that $t_A=e^{\tilde{\sigma}(\theta (c((\log(r))^2)))}:\widehat{K\pi} \to \widehat{K \pi}$.
\end{rem}

\end{document}